\documentclass{amsart}
\usepackage{amssymb,amsthm,color, amscd}
\usepackage{mathtools}
\renewcommand{\leq}{\leqslant}

\renewcommand{\emptyset}{\varnothing}
\newcommand{\rk}{\mathop{\mathrm{rk}}\nolimits}

\newcommand{\End}{\mathop{\mathrm{End}}\nolimits}
\newcommand{\Ker}{\mathop{\mathrm{Ker}}\nolimits}

\newcommand{\Hom}{\mathop{\mathrm{Hom}}\nolimits}

\def\Q{{\mathbb{Q}}}
\def\Z{{\mathbb{Z}}}
\def\N{{\mathbb{N}}}

\newcommand{\bbZ}{\mathbb{Z}}

\newcommand{\calG}{\mathcal{G}}

\newcommand{\Image}{\operatorname{Im}}

\theoremstyle{plain}
\newtheorem{theorem}{Theorem}[section]
\newtheorem{lemma}[theorem]{Lemma}

\newtheorem{Cexample}[theorem]{Corner-Example}
\newtheorem{Summary}[theorem]{Summary}
\newtheorem{question}[theorem]{Question}
\newtheorem{proposition}[theorem]{Proposition}

\theoremstyle{definition}
\newtheorem{definition}[theorem]{Definition}

\newtheorem{corollary}[theorem]{Corollary}

\date{\today}
\title{Decompositions of torsion-free abelian groups}

\author{G\'{a}bor Braun}

\address{ISyE, Georgia Institute of Technology\\
755 Ferst Drive\\ NW
Atlanta\\ GA 30332
USA\\}

\email{gabor.braun@isye.gatech.edu}   
\author {Phill Schultz}
\address {School of Mathematics and Statistics\\
The University of Western Australia \\ Nedlands\\
 Australia,  6009}
\email {phill.schultz@uwa.edu.au}

\author{Lutz Str\"ungmann}
\thanks{The third author is grateful to the German Research Foundation for their support by the project DFG-STR 627/14-1}
\address{Faculty of Computer Sciences\\ University of Applied Science Mannheim\\68163 Mannheim\\Germany}
\email{l.struengmann@hs-mannheim.de}


\begin{document}
\maketitle

\large
\begin{abstract}\noindent
  It is known that every torsion-free abelian group of finite rank has
  a maximal completely decomposable summand that is unique up to
  isomorphism. We show that groups of infinite rank need not have
  maximal completely decomposable summands, but when they do, this
  summand is unique up to isomorphism.
\end{abstract}

\section{Introduction}

It is well-known that torsion-free abelian groups of finite rank are
highly complicated and that their complexity increases with the rank
by an important result due to Simon Thomas \cite{thomas}.
Thus there is no hope for a classification except for the torsion-free
groups of rank $1$ where Reinhold Baer \cite{Baer} gave a complete
description in terms of types.
Clearly, in the infinite rank case the situation is even more
difficult.
As a consequence it is of great interest to investigate direct
decompositions of torsion-free groups into indecomposable summands and
to try to describe the indecomposable bricks.
Such decompositions always exist in the finite rank case and a result
due to Lady (\cite{Lady}) says that for finite rank torsion-free
groups there are only finitely many such decompositions up to
isomorphism.
Collecting the summands of rank $1$ one obtains a maximal completely
decomposable summand of a torsion-free group of finite rank and a
classical result by Stein says that this can be even done in the
infinite rank case when restricting to summands of type the
integers.
In this paper we are investigating a generalisation of the
corresponding results to the infinite rank case.

The principal  result in \cite{MaSc}, states that if $G$ is a torsion-free abelian group of finite rank, then $G$ has a decomposition $A\oplus H$ in which $A$ is completely decomposable and $H$ is clipped, that is, $H$ has no rank 1 summand. In this case, $A$ is unique up to isomorphism and $H$ is unique up to {near isomorphism}, an equivalence of abelian groups which is explained below.

Our extension states that for any torsion-free abelian group $G$, if
$G=A\oplus H$ where $A$ is completely decomposable and $H$ is clipped,
then $A$ is unique up to isomorphism and the rank of $H$ is uniquely
determined.  Furthermore, we find classes of groups for which such a
decomposition exists and provide examples where it does not exist.
 
Throughout, we denote by $\calG$ the category of torsion-free abelian
groups.  Unexplained notation comes from the standard reference
\cite{Fuchs}.  For details on almost completely decomposable groups we
refer to the book \cite{Mader} and for torsion-free abelian groups of
finite rank in general to \cite{Arnold}.

\section{Maximal completely decomposable summands of
  torsion-free abelian groups}

A torsion-free abelian group $H$ is {\it clipped} if it has no direct summand of rank $1$ and $\tau$-clipped if it has no rank $1$ direct summand of type $\tau$. The following theorem is proved in \cite{MaSc}. Recall that groups  $G$ and $G'\in\calG$ of finite rank are called {\it nearly-isomorphic} if for every prime $p$ there is a monomorphism $\varphi_p: G \rightarrow G'$ such that $G'/\varphi_p(G)$ is finite and relatively prime to $p$. 

\begin{theorem}\cite{MaSc}
\label{finite-rank}
Let $G$ be a torsion-free abelian group of finite rank. Then $G$ has a Main Decomposition $G=A \oplus H$ such that $A$ is completely decomposable and $H$ is clipped.\\ Moreover, if $G=A' \oplus H'$ is a second decomposition with $A'$ completely decomposable and $H'$ clipped, then $A$ and $A'$ are isomorphic while $H$ and $H'$ are nearly isomorphic.
\end{theorem}

Note that the completely decomposable summand $A$ of $G$ in the above Theorem \ref{finite-rank}  is maximal in the obvious sense and that its existence follows from the fact that $G$ is of finite rank, so $A$ can just be chosen as a completely decomposable summand of maximal rank. This will be the main obstacle in our generalisation  to $\calG$, since
even in the countable case,
the existence of such a maximal completely decomposable summand is
by no means obvious.

We define Main Decompositions as in the finite rank case.

\begin{definition}
  \label{main-definition}
  Let $G\in\calG$.
  A {\it Main Decomposition} of $G$
  is a decomposition $G=A \oplus B$
  in which $A$ is completely decomposable and $B$ is clipped.
\end{definition}

We begin with the uniqueness result.

\subsection{Uniqueness of Main Decompositions}
In this section we will prove uniqueness of the Main Decomposition,
when it exists, of all $G\in\calG$.

\begin{theorem}
  \label{main1}
  Assume that $G\in\calG$ decomposes as $G=A \oplus H$
  where $A$ is completely decomposable and $H$ is clipped.
  Then $A$ is unique up to isomorphism.
\end{theorem}

In order to prove Theorem \ref{main1} we need   a few definitions and lemmas.  

\begin{definition}
Let $G$ and $G'\in\calG$.  Then
\begin{itemize}
\item $G$ and $G'$ are {\it mono-equivalent} (denoted $G \sim_{mono} G'$)
  if there are monomorphisms $\varphi:G \rightarrow G'$ and $\psi:G'
  \rightarrow G$ (\cite[Chapter 9]{Facchini});
\item a type $\tau$ is called {\it extractable for $G$} if $G$ has a rank one summand of type $\tau$;
\item $G$ is called {\it strongly separable} if every pure subgroup of rank one is a direct summand.
\end{itemize}
\end{definition}

Note that a homogeneous completely decomposable group is strongly separable and it is obvious that two homogeneous completely decomposable groups are mono-equivalent if and only if they are isomorphic since mono-equivalence implies that the groups must have the same rank. 

However, there are easy examples where mono-equivalence does not imply
isomorphism, e.g., $G=\Z \oplus \bigoplus_{\omega}\Q$
and $G'=\bigoplus_{\omega}\Q$ are mono-equivalent but not isomorphic.
Furthermore, for completely decomposable groups
the set of extractable types is exactly the set of
critical types (\cite[Definition 2.4.6]{Mader}).

Before we proceed with proving uniqueness of Main Decompositions we classify the strongly separable groups. Let us note that there are large classes of strongly separable groups, e.g. all dual groups and more generally any homogeneous separable group (see \cite[Lemma 4.10, page 505] {Fuchs} and \cite[Lemma 4.5, page 503]{Fuchs}).

\begin{proposition}
  Let $G\in\calG$ have maximal divisible subgroup $D$.
  $G$ is strongly separable if and only if
  $G/D$ is separable and homogeneous.
\end{proposition}

\begin{proof}
Clearly, $G$ is strongly separable if and only if
$G/D$ is strongly separable,
so we may assume that $G$ is reduced.
Obviously, any separable and homogeneous group is strongly separable.

Conversely, assume that $G$ is a strongly separable group. Then $G$ has to be homogeneous by Lemma 2.2. from \cite{HH}.
\end{proof}

We remark that for separable groups with linearly ordered typeset we at least get that every pure subgroup of finite rank is a quasi-summand by \cite{Cornelius}. These groups are called {quasi-separable}.

\begin{lemma}
  \label{lem-homogeneous}
  Let $G\in\calG$.
  Suppose that $G=D \oplus B = A \oplus C$
  where both $B$ and $C$ are $\tau$-clipped
  for all extractable types of $A$ and $D$.
  Moreover, assume that
  $A$ and $D$ are strongly separable torsion-free groups.
  Then  $D\sim_{mono} A$.
\end{lemma}

\begin{proof}
We need to show that there are two monomorphisms mapping $D$ into $A$
and $A$ into $D$ respectively.  Let $\alpha: D \rightarrow A$ be the
projection along $C$.  We claim that $\alpha$ is injective.  Assume
that $x\alpha=0$ for some non-zero $x \in D$.  By strong separability
of D, $\left< x \right>_*$ is a direct summand of $D$, so the type
$\tau$ of $x$ is an extractable type of $D$ and hence $C$ is
$\tau$-clipped.  However, $x\alpha=0$ implies that $\left< x \right>_*
\subseteq ker(\alpha)=C$ which is then a summand of $C$, a
contradiction.  Thus $\alpha$ is monic and by symmetry also the
projection $\delta:G \rightarrow D$ along $C$ is monic.
Hence $A \sim_{mono} D$.
\end{proof}

We have an immediate corollary.

\begin{corollary}
  \label{cor-homogeneous}
  Let $G\in\calG$ have decompositions $G=D \oplus B = A \oplus C$
  where $B$ and $C$ are $\tau$-clipped
  and $A$ and $D$ are $\tau$-homogeneous completely decomposable
  for some type $\tau$.
  Then $A \cong D$.
\end{corollary}

\begin{proof}
It suffices to note that $A$ and $D$ are strongly separable
and hence Lemma \ref{lem-homogeneous} gives that $A \sim_{mono} D$.
But this implies already that $A \cong D$.
\end{proof}

\begin{lemma}
  \label{lem-inhomogeneous}
  Let $G\in\calG$ have decomposition $G=A \oplus B$ where $A$ is
  completely decomposable and both \(A\) and $B$ are $\tau$-clipped.
  Then $G$ is $\tau$-clipped as well.
\end{lemma}

\begin{proof}
This follows immediately from \cite[Lemma 2.4]{MaSc} which is stated
for finite rank groups, but the proof does not use that \(B\) is of
finite rank.
We just have to notice that $G$ is $\tau$-clipped if and only if
$G'=A' \oplus B$ is $\tau$-clipped for every finite rank summand $A'$
of $A$.
\end{proof}

We are now ready to prove Theorem \ref{main1}.
\begin{proof}[Proof of Theorem \ref{main1}]
Assume that $G\in\calG$ has a decomposition $G=A \oplus H$ where $A$ is completely decomposable and $H$ is clipped. We have to show that $A$ is unique up to isomorphism. 

Assume that there is a second decomposition $G=A' \oplus H'$ where
$A'$ is completely decomposable and $H'$ is clipped. Let
$A=\bigoplus\limits_{\tau}A_{\tau}$ and
$A'=\bigoplus\limits_{\tau}A'_{\tau}$ be the corresponding
decompositions of $A$ and $A'$ into $\tau$-homogeneous summands. Fix
an arbitrary type $\sigma$. We rearrange the summands and consider
$G=A_{\sigma} \oplus \left( \bigoplus\limits_{\tau
    \not=\sigma}A_{\tau} \oplus H \right)= A'_{\sigma} \oplus \left(
  \bigoplus\limits_{\tau \not=\sigma}A'_{\tau} \oplus H' \right)$. By
Lemma \ref{lem-inhomogeneous} the groups $ \bigoplus\limits_{\tau
  \not=\sigma}A_{\tau} \oplus H $ and $  \bigoplus\limits_{\tau
  \not=\sigma}A'_{\tau} \oplus H' $ are both $\sigma$-clipped.
Hence $A_{\sigma} \cong A'_{\sigma}$ by Corollary
\ref{cor-homogeneous}.
Since this holds for all $\sigma$, we obtain $A \cong A'$.
\end{proof}

Known extensions of the concept of  near-isomorphism from finite rank
groups to $\calG$, for example \cite{neariso} are not strong enough to
show uniqueness of the clipped components in Main Decompositions up to
near isomorphism.
Instead, we define groups $G$ and $G'\in\calG$ to be
\textit{sufficiently isomorphic}, denoted $G\cong_{suf}G'$ if there
exists a completely decomposable group $C$ such that $G\oplus C\cong
G'\oplus C$.
This is motivated by the classical result that for almost
completely decomposable abelian groups of finite rank being
nearly-isomorphic is equivalent to being sufficiently isomorphic
(see \cite[Lemma 9.1.10]{Mader}).
Thus the clipped components in Main Decompositions are sufficiently
isomorphic but properties of sufficiently isomorphic groups are yet to
be investigated. Nevertheless, we have the following result.

\begin{lemma}
\label{mono}
 Let $G\in\calG$ such that $G=A \oplus H$ where $A$ is homogeneous completely decomposable and $H$ is clipped. If $G$ has a second Main Decomposition $G=A' \oplus H'$ with $A \cong A'$, then $H $ and $ H'$ are mono-equivalent.

\end{lemma}

\begin{proof}
We claim that the canonical projections $\pi_H:G \rightarrow H$ and
$\pi_{H'}: G \rightarrow H'$ restricted to $H'$ and $H$ respectively,
are monic. This will then imply that $H \sim_{mono} H'$. Assume that
$\pi_H(h')=0$ for some $0 \not= h' \in H'$. Then $h' \in A$. Since $A$
is homogeneous completely decomposable we conclude that $\left< h'
\right>_*$ is a direct summand of $A$. So we have $G=A_1 \oplus \left<
h' \right>_* \oplus H$ where $A=A_1 \oplus \left< h' \right>_*$. By
the modular law we conclude $H'=H' \cap G = \left< h' \right>_* \oplus
H' \cap (A_1 \oplus H)$ and hence $H'$ has a direct summand of rank
$1$ contradicting the assumption that $H'$ is clipped. Therefore,
$\pi_H \restriction H'$ and by symmetry also $\pi_{H'} \restriction H$
are monic which shows $H \sim_{mono} H'$ (note that $A'$ is also
homogeneous completely decomposable since $A \cong A'$ by Theorem
\ref{main1}).
\end{proof}

Note that in Lemma \ref{mono} the completely decomposable
part is assumed to be homogeneous.  It will be shown in Theorem
\ref{rank complement} that in general at least the rank of the clipped
component in any Main Decomposition is uniquely determined.
Nevertheless, mono-equivalence gives a little more in the above Lemma
\ref{mono}.  We close this section with an open question.

\begin{question}
  Given two Main Decompositions $G=A \oplus H=A' \oplus H'$ with $A$
  and $A'$ completely decomposable and $H, H'$ clipped.
  Is it true that
  $H$ is indecomposable if and only if $H'$ is indecomposable?
\end{question}

\subsection{Existence of Main Decompositions}

It remains to investigate for torsion-free abelian groups of infinite
rank whether we can find a completely decomposable summand such that its
complement is clipped.  We start by remarking that we cannot hope for
a positive answer in general for groups of arbitrary size.  For
example, if $\kappa$ is an infinite cardinal, then the Specker group
$\bbZ^{\kappa}$, and more generally any vector group (\cite[Section
96]{Fuchs}) of rank $>\kappa$ has no clipped direct summand.

However, there are clearly many torsion-free groups of infinite rank
that admit a Main Decomposition, for example every indecomposable
group of rank $>1$ and every super decomposable group is clipped.

Let us note that it is not even known if groups that have the
Krull-Schmidt property, i.e. groups that have up to isomorphism only
one decomposition into indecomposable summands, need to have a Main
Decomposition.

We now state a first positive result.   The {\it length} of a decomposition of a torsion-free group $G$ is defined to be the number of summands.

\begin{proposition}
\label{finite-endo-rank}
Let $G$ be a torsion-free abelian group of infinite rank for which the length of any direct decomposition into indecomposable summands is bounded by some fixed integer. Then $G$ has a Main Decomposition  $G=A \oplus H$ where $A$ is completely decomposable of finite rank and $H$ is clipped.
\end{proposition}

\begin{proof}
Let $n$ be an upper bound on the length of decompositions of \(G\).
For each decomposition into indecomposable summands, say $G=G_1 \oplus
\dotsb \oplus G_{n_i}$, put $A_i=\bigoplus\limits_{\rk(G_j)=1}G_j$ the
direct sum of all rank \(1\) summands.
Clearly, $A_i$ is of finite rank.
Now choose $A$ among $A_1, A_2, \dotsc $ of maximal (finite) rank.
Then obviously $G=A \oplus H$ where $H$ is clipped.
\end{proof}
  
  \begin{corollary}\label{finite rank End}
  Let $G\in\calG$
  such that its endomorphism ring $\End(G)$ has finite rank.
  Then $G$ has a Main Decomposition $A\oplus H$
  where $A$ is completely decomposable of finite rank
  and $H$ is clipped.
 \end{corollary}
  
\begin{proof}
Under the assumptions $G$ has only finitely many non-isomorphic
decompositions into indecomposable summands and each of these
decompositions is finite, i.e. is a finite direct sum of
indecomposable summands.  This follows as in Lady's proof for the
finite rank case using Lemmas 6.1 and 6.8 in \cite[page 448]{Fuchs}.
Now Proposition \ref{finite-endo-rank} applies.
\end{proof}

It is an open question whether the clipped complement in a Main
Decomposition is unique in any sense but the next lemma shows that at
least its rank is unique.

\begin{theorem}
  \label{rank complement}
  Let $G$ be an infinite torsion-free abelian group and
  $G=A \oplus H$ a Main Decomposition
  with $A$ completely decomposable and $H$ clipped.
  Then in any Main Decomposition $G=A' \oplus H'$ of $G$
  with $A'$ completely decomposable and $H'$ clipped,
  the rank of $H'$ is equal to the rank of $H$.
\end{theorem}

\begin{proof}
Let $\kappa=\rk(H)$.
We may assume for a contradiction
without loss of generality that \(\rk(H') > \kappa\).
Let us assume for now that \(\kappa\) is infinite.
By assumption there is a summand $A_1'$ of $A'=A_1' \oplus A_2'$ of
rank at most $\kappa$ such that $H$ is a summand of $A_1' \oplus H'$,
say $A_1' \oplus H'=H \oplus C$.  Now,
\[A \cong G/H \cong (A_1' \oplus A_2' \oplus H')/H \cong C \oplus
  A_2'\] is completely decomposable.
It follows that $C$ is completely decomposable as well since it is a
summand of a completely decomposable group.
Since $\rk(A_1')$ is at most $\kappa$ there is a direct summand $C'$
of $C=C' \oplus C''$ of rank at most $\kappa$ such that $A_1'$ is a
direct summand of $H \oplus C'$, say $H \oplus C'=A_1' \oplus T$. It
follows that
\[ H' \cong (H \oplus C)/A_1'  = (A_1' \oplus T \oplus C'')/A_1'  = T \oplus C'' \]
and since $\rk(H') > \rk(H)$ but $\rk(T) \leq \kappa$ we conclude that
$C'' \not= \{0\}$ and thus $H'$ is not clipped - a contradiction.

For finite \(\kappa\),
replacing ``at most \(\kappa\)'' with ``finite''
in the argument above,
we obtain that \(H' = T\) has finite rank, too, and
\(A'_{1} \oplus H' = C \oplus H\) are two Main Decompositions of a
finite rank group.
By Theorem~\ref{finite-rank}, \(H\) and \(H'\) have the same rank.
\end{proof}
    
We now turn our attention to countable groups in $\calG$.
A first step is the homogeneous case. This is motivated by a classical
result due to Stein (see \cite[Corollary 8.3, p. 114]{Fuchs}) which
says that for a given countable torsion-free group $G$,
there is a decomposition $G=F \oplus H$
with $F$ free and $H$ has trivial dual, hence it is $\Z$-clipped.

\begin{proposition}
\label{Prop-homogeneous}
Let $G\in\calG$ be countable and $\tau$ a type.
Then $G(\tau)$ admits a decomposition $G(\tau)=A_{\tau} \oplus B$
where $A_{\tau}$ is $\tau$-homogeneous and $B$ is $\tau$-clipped.
\end{proposition}

\begin{proof}
 Without loss of generality we assume that $G=G(\tau)$. We now follow Stein's proof \cite[Corollary 19.3]{Fuchs} and put
\[ K(G)=\bigcap_{ 0\not = \eta : G \rightarrow \tau} \Ker(\eta) \]
which is the intersection of all non-trivial homomorphisms from $G$ into a rational group of type $\tau$. Then 
\[ \alpha: G/K(G) \rightarrow \tau^{\Hom(G,\tau)}
  \quad \text{ with } \quad
  g \mapsto (\eta(g) : \eta \in \Hom(G,\tau)) \]
is a monomorphism from $G/K(G)$ into the elementary vector group
$\tau^{\Hom(G,\tau)}$.
We claim that
\begin{enumerate}
\item every element $x \in G \backslash K(G)$ has type $\tau$;
\item the image $\Image(\alpha)$ of $\alpha$ is
  $\tau$-homogeneous completely decomposable.
\end{enumerate}
Since $K(G)$ is pure in $G$ it then follows from Baer's lemma (see
\cite[Proposition 3.8, page 426]{Fuchs}) that $K(G)$ is a summand with
$\tau$-homogeneous completely decomposable complement $B$. Obviously
$K(G)$ is $\tau$-clipped which gives the desired Main Decomposition
$G=K(G) \oplus B$.

It remains to prove $(1)$ and $(2)$. However, $(1)$ is clear since any
element in $G$ has type greater than or equal to $\tau$ by assumption
and clearly all elements of type greater than $\tau$ are in the kernel
of any $\eta : G \rightarrow \tau$.

We still have to prove that $K(G)$ is a summand of $G$.
Since $G=G(\tau)$
we have that $Im(\alpha)$ satisfies the same property, i.e.
$(\Image(\alpha))(\tau) = \Image(\alpha)$.
Since $\tau^{\Hom(G,\tau)}$ certainly has no type greater than $\tau$
we conclude that $\Image(\alpha)$ is $\tau$-homogeneous.
Write $\Image(\alpha)$ as
the union of an ascending chain $\Image(\alpha)=\bigcup_{n}A_n$
of pure finite rank subgroups.
Then each of the $A_n$ is completely decomposable
since the vector group $\tau^{\Hom(G,\tau)}$ is separable
by a result due to Mishina (see \cite[page 512]{Fuchs}).
It now follows that $\Image(\alpha)$ itself
is completely decomposable by \cite[Theorem 3.14, page 429]{Fuchs}.
\end{proof}

We have an immediate 

\begin{theorem}\label{main-homogeneous}
Let $G$ be a $\tau$-homogeneous countable torsion-free group. Then $G$ has a Main Decomposition.
\end{theorem}

\begin{proof}
Since $G=G(\tau)$ we just apply Proposition \ref{Prop-homogeneous}.
\end{proof}

\begin{corollary}\label{flotypeset}
Let $G$ be a countable torsion-free group whose set of extractable types is finite and linearly ordered. Then $G$ has a Main Decomposition.
\end{corollary}

\begin{proof}
The proof is an easy induction on the length of the chain of extractable types. The case of length $1$ is settled by Theorem \ref{main-homogeneous}. \end{proof}

To summarize, we have shown the following on the existence of Main Decompositions. 

\begin{Summary}
A torsion-free abelian group $G$ admits a Main Decomposition in the following cases:
\begin{enumerate}
\item Lengths of direct decompositions are bounded.  (Proposition \ref{finite-endo-rank})
\item $\End(G)$ has finite rank. (Corollary \ref{finite rank End})
\item $G$ is countable homogeneous. (Theorem \ref{main-homogeneous})
\item $G$ is countable with finite linearly ordered extractable typeset. (Corollary \ref{flotypeset}) 
\end{enumerate}

\end{Summary}
We will see in Example \ref{Example-Corner-3} that there are countable
groups in $\calG$ that do not possess a Main Deocmposition.
However, the general question on the existence of Main Decompositions
remains open.

\begin{question}
  Can we characterise those (countable) torsion-free abelian groups
  $G$
  such that there exists a completely decomposable summand $A$
  with $G=A \oplus C$ and $C$ clipped?
\end{question}

In the remaining part of this section we study some examples constructed by A.L.S. Corner in order to show Main Decompositions for non-homogeneous torsion-free groups and in order to give an example of a countable torsion-free abelian group that does not possess a Main Decomposition. Note that all the examples are based on a completely decomposable subgroup that has an infinite anti-chain of types as critical typeset.\\

The first example shows that a countable group may have a non-trivial Main Decomposition and at the same time can be decomposed into two indecomposable summands of countably infinite rank.

\begin{Cexample}\cite[Theorem 1.2, page 482]{Fuchs}
\label{Example-Corner-1}
There exists a torsion-free group $G$ of countable rank that has two decompositions
\[ G=A \oplus B = C \oplus D  \]
where $B, C$ and $D$ are indecomposable of rank $\aleph_0$ and $A$ is completely decomposable. Hence $G=A \oplus B$ is a Main Decomposition of $G$.

\end{Cexample}

\begin{proof}
The group is constructed using
\begin{itemize}
\item $A=\left< p_n^{-\infty} a_n : n \in \N \right>$
\item $B=\left< p_n^{-\infty} b_n, p^{-1}q^{-1}(b_n-b_{n+1}) :  n \in \N \right>$
\item $C=\left< p_n^{-\infty}c_n, p^{-1}(c_n-c_{n+1}) : n \in \N  \right>$ 
\item $D=\left< p_n^{-\infty}d_n, q^{-1}(d_n-d_{n+1}) : n \in \N  \right>$ 
\end{itemize}
where the numbers $p_n, p, q$ are distinct primes and the $a_n, b_n$ are linearly independent base elements of a $\Q$-vector space. Moreover, the $c_n,d_n$ are chosen as
\[ c_n=p a_n + tb_n \quad \text{and} \quad d_n=qa_n + sb_n  \]
where the $t,s$ are integers such that \(p s - t q = 1\),
see \cite[Theorem 1.2, page 482]{Fuchs} for more details.
\end{proof}

The next example shows that the infinite direct sum of indecomposable rank $2$ groups can be at the same time a direct sum of two indecomposable groups of infinite rank and still have a Main Decomposition.
 
\begin{Cexample}\cite[Theorem 1.1, page 481]{Fuchs}
\label{Example-Corner-2}
There exists a torsion-free group $G$ of countable rank that has two decompositions
\[ G=B \oplus C = \bigoplus\limits_{n \in \Z} E_n \]
where $B$ and $C$ are indecomposable of rank $\aleph_0$ and $E_n$ are indecomposable of rank $2$.

\end{Cexample}

\begin{proof}
The group is constructed using
\begin{itemize}
\item $B=\left< p_n^{-\infty} b_n, q_n^{-1}(b_n+b_{n+1}) :  n \in \Z \right>$
\item $C=\left< p_n^{-\infty} c_n, r_n^{-1}(c_n+c_{n+1}) :  n \in \Z \right>$
\item $E_n=\left< p_n^{-\infty}u_n, p_{n+1}^{-\infty}v_{n+1}, q_n^{-1}r_n^{-1}(u_n+v_{n+1}) \right>$ for $n \in \Z$
\end{itemize}
where the $p_n, q_n, r_n$ are distinct primes and the $b_n, c_n$ are linearly independent base elements of a $\Q$-vector space. Moreover, the $u_n,v_n$ are chosen as
\[ u_n=(1+k_n)b_n-k_nc_n \quad \textit{and} \quad v_n=k_nb_n + (1-k_n)c_n \]
where the $k_n$ are specified as in \cite[Theorem 1.1, page 481]{Fuchs}.
\end{proof}

We now prove that the group from Example \ref{Example-Corner-2} does have a Main Decomposition.

\begin{lemma}
The group $G$ from Example \ref{Example-Corner-2} has a Main Decomposition.
\end{lemma}

\begin{proof}
We claim that $G$ has a decomposition of the form 
\[ G= \bigoplus\limits_{n \in \Z} \left< p_n^{\infty}z_n \right> \oplus \left< p_n^{\infty}t_n, \frac{t_n+t_{n+1}}{q_nr_n}: n \in \Z \right> \]
where the elements $t_n, z_n \in G$ are chosen appropriately for $n \in \Z$. It is then easy to see that the second summand $H:= \left< p_n^{\infty}t_n, \frac{t_n+t_{n+1}}{q_nr_n}: n \in \Z \right>$ is indecomposable (thus clipped) and hence $G$ possesses a Main Decomposition.
Define for $n \in \Z$ the following elements 
\[ t_n=(1+\alpha_n)b_n - \alpha_n c_n \quad \textit{ and }\quad z_n= \alpha_n b_n - (1- \alpha_n)c_n \]
where $\alpha_n \equiv 0$ modulo $q_{n-1}q_n$ and $\alpha_n \equiv -1$ modulo $r_{n-1}r_n$ for all $n \in \Z$. Note that this choice is possible.
Clearly, each $z_n$ is infinitely many times divisible by $p_n$ since the $b_n$ and $c_n$ are so. Moreover, 
\[ t_n + t_{n+1}= (1+\alpha_n)b_n - \alpha_nc_n + (1+\alpha_{n+1})b_{n+1} - \alpha_nc_{n+1} \]
and thus
\[ t_n + t_{n+1} \equiv b_n + b_{n+1} \mod q_n \]
and 
\[ t_n + t_{n+1} \equiv c_n+c_{n+1} \mod r_n \]
which implies that $t_n + t_{n+1}$ is divisible by $q_nr_n$ inside $G$ for all $n \in \Z$. It is now readily verified that $\bigoplus\limits_{n \in \Z} \left< p_n^{\infty}z_n \right> \oplus \left< p_n^{\infty}t_n, \frac{t_n+t_{n+1}}{q_nr_n}: n \in \Z \right> $ is indeed a direct sum equal to $G$.
\end{proof}

Finally, we give a last example where we shall prove that the group does not possess a Main Decomposition. Note that this group has continuously many non-isomorphic decompositions.

\begin{Cexample}\cite[Theorem 1.3, page 483]{Fuchs}
\label{Example-Corner-3}
  There exists a group $G\in\calG$ of countable rank
  such that for every choice of a sequence $r_1, \dotsc, r_n, \dotsc$
  of positive integers with infinitely many $r_n > 1$,
  there exist indecomposable groups $G_n$ of rank $r_n$
  such that $G=\bigoplus\limits_{n \in \N} G_n$.
\end{Cexample}

\begin{proof}
The group is constructed as $G=\bigoplus\limits_{n \in N} B_n$ using
\begin{itemize}
\item $B_n=\left< p^{-\infty} u_n, p_n^{-\infty}x_n, q_n^{-1}(u_n+x_n) :  n \in \N \right>$
\end{itemize}
where $p_n, q_n, p$ are distinct primes and the $u_n, x_n$ are
linearly independent base elements of a rational vector space.  The
different decompositions are obtained by rearranging summands and
applying the construction from \cite[Theorem 5.2, page 439]{Fuchs}.
\end{proof}

It was already noted by Corner in \cite{Corner} that the example is
best possible in the sense that the group $G$ cannot have
a Main Decomposition with finite rank clipped part.  This
also follows from our Lemma \ref{rank complement}.  However, we prove
next that the group from Example \ref{Example-Corner-3} does not
possess a Main Decomposition at all.

\begin{theorem}
  The group $G$ from Example \ref{Example-Corner-3}
  is a countable group that does not posses a Main Decomposition.
\end{theorem}

\begin{proof}
Assume that $G$ has a Main Decomposition $G=A \oplus H'$
with $A$ completely decomposable and $H'$ clipped.
By Lemma \ref{rank complement} and Corner's remark
we must have that the rank of $H'$ is infinite.
However, we prove next that
every direct summand of \(G\) is a direct sum of finite-rank
indecomposable groups of the form
  \begin{equation}
    \label{eq:1}
    B_{v, S, \alpha} \coloneqq
    \left<
      p^{-\infty} v, p_{n}^{-\infty} x_{n},
      q_{n}^{-1} (v + \alpha_{n} x_{n}) : n \in S
    \right>
  \end{equation}
  for some finite set \(S \subseteq \mathbb{N}\), \(v \in G\)
  and integers \(\alpha_{n}\).
  
  It then follows as in \cite[Theorem 5.2, page 439]{Fuchs} that a combination of the summands $B_{v, S, \alpha}$ always have summands of rank $1$ and hence $H'$ is not clipped, a contradiction.

 We now prove statement \eqref{eq:1}.
 Let \(H\) be a direct summand of \(G\) (not necessarily clipped), and let \(H_{p}\) denote the
\(p\)-divisible part of \(H\).
By standard arguments,
using that \(\left< p^{-\infty} u_{n} : n = 0, 1, \dotsc \right>\)
and \(\left< p_{n}^{-\infty} x_{n} \right>\) are fully invariant,
\(H\) has a presentation of the form
\begin{equation}
  \label{eq:2}
  H \coloneqq 
  \left<
    p^{-\infty} v_{i}, p_{n}^{-\infty} x_{n},
    q_{n}^{-1} (w_{n} + x_{n})
    : i \in I, n \in S
  \right>
\end{equation}
for some possibly infinite \(S \subseteq \mathbb{N}\),
where the \(w_{n}\) are contained in \(H_{p}\).
Obviously, \(H_{p} = \left< p^{-\infty} v_{i} : i \in I \right>\).

We decompose \(H\) into indecomposable summands as follows.
We enumerate the elements of \(S\) as
a sequence \(m_{1}, m_{2}, \dotsc\).
First we find a direct summand \(B_{v_{1}, T_{1}, \alpha_{1}}\)
of \(H\) with \(m_{1} \in T\), as explained later.
Then we find a direct summand \(B_{v_{2}, T_{2}, \alpha_{2}}\)
in the complement containing the next \(m_{i}\) not contained in
\(T_{1}\).
We continue this process until the sequence \(S\) is exhausted,
this will happen after at most countable many steps.
As a result, we obtain a possibly infinite sequence of subgroups
\(B_{v_{1}, T_{1}, \alpha_{1}}, B_{v_{2}, T_{2}, \alpha_{2}}, \dotsc\)
such that their direct sum
\(B \coloneqq \bigoplus_{k=1, \dotsc} B_{v_{k}, T_{k}, \alpha_{k}}\)
is a pure subgroup of \(H\)
containing all the \(x_{n}\) for \(n \in S\),
and hence \(H = B + H_{p}\).
Therefore, \(B \cap H_{p}\) is pure
in the homogeneous, completely decomposable group \(H_{p}\),
and hence it is also a direct
summand: \(H_{p} \coloneqq B \oplus C\) with \(C\) completely
decomposable.
It follows that \(H = B \oplus C\) producing the desired decomposition
of \(H\).
(The rank \(1\) summands of \(C\) have the form
\(B_{v, \emptyset, ())}\).)

Finally, we return to finding the direct summands
\(B_{v_{k}, T_{k}, \alpha_{k}}\) of \(H\).
Given an arbitrary \(m \in S\),
we shall find a direct summand \(B_{v, T, \alpha}\) of \(H\)
for some finite set \(T \subseteq S\) containing \(m\).

Choose \(v_{0} = w_{m} / k\) for some positive integer \(k\)
such that \(\left< p^{-\infty} v \right>\) is pure in \(H_{p}\).
In particular,
\begin{equation}
  \label{eq:3}
  H_{p} =
  \bigoplus_{k=0}^{N} \left< p^{-\infty} v_{k} \right>
  ,
\end{equation}
where \(N\) is a nonnegative integer or infinity.
Now the \(w_{n}\) have the form
\begin{equation}
  \label{eq:4}
  w_{n} = \sum_{i=1}^{k_{n}} \beta_{n, i} v_{i}
\end{equation}
with \(\beta_{n, i} \in \mathbb{Z}[1/p]\)
and without loss of generality, \(\beta_{n,k_{n}}\) is not divisible
by \(q_{n}\).
We claim that there are only finitely many \(n\) with \(k_{n} \leq l\)
for any nonnegative integer \(l\).
Indeed, writing the \(v_{i}\) as linear combination of the \(u_{n}\),
and considering \(w_{n} - u_{n}\), which is divisible by \(q_{n}\),
we conclude that the coefficient of \(u_{n}\) in \(w_{n} - u_{n}\)
must be divisible by \(q_{n}\), which is only possible if \(u_{n}\)
appears with non-zero coefficient in at least one of \(v_{0}, v_{1},
\dotsc, v_{k_{n}}\).
Therefore \(k_{n} \leq l\) is only possible for the finitely many
\(n\) for which \(u_{n}\) has a non-zero coefficient in some of
\(v_{0}, \dotsc, v_{l}\).

Let \(T\) be the set of \(n \in S\) with \(k_{n} = 0\),
which includes \(n = m\) by the choice of \(v_{0}\),
and is finite by the previous paragraph.
Let \(\alpha_{n}\) be the multiplicative inverse of \(\beta_{n, 0}\)
modulo \(q_{n}\) for \(n \in T\),
thus \(v_{0} + \alpha_{n} x_{n}\)
is divisible by \(q_{n}\).

Now we show that \(B_{v_{0}, T, \alpha}\) is a direct summand of \(H\)
by defining a splitting map
\(\varphi \colon H \to B_{v_{0}, T, \alpha}\)
to the inclusion of \(B_{v_{0}, T, \alpha}\) into \(H\).
We set \(v_{0} \varphi = v_{0}\) and
\(x_{n} \varphi = x_{n}\) for \(n \in T\).
Let \(x_{n} \varphi = 0\) for \(n \notin T\).
We define the \(v_{k} \varphi \in p^{-\infty} \left< v_{0} \right>\)
inductively, subject to the condition
that \(\sum_{i=1}^{k} \beta_{n, i} (v_{i} \varphi)\)
is divisible by \(q_{n}\) for all \(n \in S \setminus T\)
with \(k_{n} = k\), which is an explicit statement
of \((w_{n} + x_{n}) \varphi\) being divisible by \(q_{n}\).
(This condition is vacuous for \(k = 0\).)
By the Chinese Remainder Theorem, \(v_{k} \phi\) can be chosen to
satisfy this condition, because there are finitely many \(n\) with
\(k_{n} = k\)
as shown above, and for these \(n\) the coefficient \(\beta_{n,k}\) is
not divisible by \(q_{n}\).

This defines \(\varphi\) on the linearly independent generators
\(v_{i}\) and \(x_{n}\) of \(H\), from which it uniquely extends to
\(H\).
By the definition it is obvious that the image of \(\varphi\) is
\(B_{v_{0}, T, \alpha}\).

\end{proof}


\begin{thebibliography}{99}
\bibitem
[Arnold 1982]{Arnold}
\textit{D. Arnold}, Finite Rank Torsion Free Abelian Groups and Rings,
Lecture Notes in Mathematics, Vol. 931, Springer Verlag, 1982.

\bibitem[Baer 1937]{Baer}\textit{R. Baer}, Abelian groups without elements of finite order, Duke Math. Journal \textbf{1}, pp. 68--122, 1937.
\bibitem[Blagoveshchenskaya and Str\"ungmann 2007]{neariso}
  \textit{E. Blagoveshchenskaya and L. Str\"ungmann},
Near-isomorphism for a class of infinite rank torsion-free abelian
groups, Comm. in Algebra, 35, pp. 1-18, 2007.


\bibitem[Cornelius 1971]{Cornelius}\textit{E. Cornelius},
  A Generalization of separable groups, Pac. J. of Mathematics, Vol. 39, pp. 603-613, 1971.

\bibitem[Corner 1961]{Corner}\textit{A. L. S. Corner},
  A note on rank and direct decomp[ositions of torsion-free abelian groups, Proc. Cambridge Philos. Soc., 57, (1961), 230--233,.


\bibitem[Facchini 1998]{Facchini} \textit{A. Facchini}, Module Theory,
Birkh\"auser Verlag, 1998.

\bibitem[Fuchs 2015]{Fuchs}
  \textit{L. Fuchs},
  Abelian Groups, Springer Monographs in Mathemaics, Springer 2015.


\bibitem[Hsiang and  Hsiang 1961]{HH}\textit{W. C.   Hsiang and  W. Y.   Hsiang}, Those abelian groups characterised by  their completely decomposable subgroups of finite rank, Pac. J. of Mathematics, Vol. 11, pp. 547-558, 1961.



\bibitem[Lady 1974]{Lady} \textit{E. L. Lady},
  Summands of Finite Rank Torsion-free Abelian Groups,  J. of Algebra, 32,  (1974) 51--52


\bibitem[Mader 2000]{Mader}
\textit{A. Mader}, Almost completely decomposable abelian groups,
Gordon and Breach, Algebra, Logic and Applications, Vol.
13, Amsterdam, 1999.
\bibitem[Mader and Schultz 2018]{MaSc}\textit{A. Mader and P. Schultz} Completely decomposable direct summands of torsion-free abelian groups of finite rank, PAMS \textbf{146}, pp. 93--96, 2018.

\bibitem[Thomas 2003]{thomas}\textit{S. Thomas} The Classification Problem for Torsion-Free Abelian Groups of Finite Rank, JAMS \textbf{16}, pp. 233--258, 2003.

\end{thebibliography}
\end{document}